\documentclass{article}%
\usepackage{amsmath}%
\usepackage{amsthm}%
\usepackage{amsfonts}%
\usepackage{amssymb}%
\usepackage{graphicx}
\usepackage{a4wide}
\usepackage{todonotes}
\usepackage[normalem]{ulem}
\usepackage[utf8]{inputenc}
\usepackage{verbatim}

\newtheorem{theorem}{Theorem}

\newtheorem{corollary}[theorem]{Corollary}

\newtheorem{lemma}[theorem]{Lemma}

\newtheorem{proposition}[theorem]{Proposition}

\DeclareMathOperator{\Aut}{Aut}

\numberwithin{theorem}{section}

\theoremstyle{definition}
\newtheorem{definition}[theorem]{Definition}

\newcommand{\grf}{\varphi}
\newcommand{\lra}{\longrightarrow}
\newcommand{\grs}{\sigma}
\newcommand{\gra}{\alpha}

\begin{document}

\title{Abelian varieties as automorphism groups of smooth projective varieties}
\author{Davide Lombardo and Andrea Maffei}

\date{}
\maketitle

\begin{abstract}
We determine which complex abelian varieties can be realized as the automorphism group of a smooth projective variety.
\end{abstract}
\section{Introduction}

In this note we determine which complex abelian varieties $A$ can be realized as the automorphism group of a complex smooth projective variety. 
{Given an abelian variety $A$, we denote by $\operatorname{Aut}_0(A)$ (respectively $\operatorname{Aut}(A)$) 
the automorphism group of $A$ as an algebraic group (respectively as a projective variety).}
We prove that if {$\operatorname{Aut}_0(A)$} is infinite then $A$ can never be realized 
{as the automorphism group of a smooth projective variety} (Theorem \ref{thm:InfiniteOrder}), while if {$\operatorname{Aut}_0(A)$} is finite there exists a smooth projective 
variety $Y$ of dimension $2+\dim A$ such that $\Aut(Y)=A$ (Theorem \ref{thm:AutoY}).

\medskip

\textbf{Acknowledgments.} This work was motivated by a more general question posed by Michel Brion in Oberwolfach; the specific case of abelian varieties, which we analyse in this paper, was raised by Corrado De Concini. 
We would like to thank Angelo Vistoli for some very useful correspondence regarding the comparison between étale and analytic cohomology.

\section{Abelian varieties with infinite automorphism group}
In this section we show that {no abelian variety with infinite 
$\operatorname{Aut}_0(A)$ can be realized as the automorphism group of a smooth projective variety:}
\begin{theorem}\label{thm:InfiniteOrder}
Let $A$ be an abelian variety such that $\operatorname{Aut}_0(A)$, the automorphism group of $A$ as an algebraic group, is infinite. 
Let $X$ be a smooth projective variety on which $A$ acts faithfully: then the automorphism group of $X$ is strictly larger than $A$.
\end{theorem}

The proof relies on the following result, due to Brion \cite[page 2]{MR2562620}:
\begin{theorem}\label{thm:Brion}
Let $X$ be a smooth projective variety on which an abelian variety $A$ acts faithfully. 
There is a positive integer $n$ and a $A[n]$-invariant closed subscheme $Y$ of $X$ such that there is an $A$-equivariant isomorphism
\[
X \cong Y \times^{A[n]} A.
\]
\end{theorem}

\begin{proof}(of Theorem \ref{thm:InfiniteOrder})
Let $\iota : A \hookrightarrow \operatorname{Aut}(X)$ be the given action of $A$ on $X$ and write $X \cong Y \times^{A[n]} A$ as in Theorem \ref{thm:Brion}. 
We can represent $X \cong Y \times^{A[n]} A$ more explicitly as the quotient
\[
X \cong \frac{Y \times A}{A[n]},
\]
where $t \in A[n]$ acts on $(y,a)$ as $t \cdot (y,a) = (\iota(t)(y),a-t)$. This quotient is well-behaved, because $A[n]$ is a finite group acting on $Y \times A$ with no fixed points. In particular, in order to give an (invertible) map $X \to X$ it is enough to give an (invertible) map $Y \times A \to Y \times A$ that is compatible with the action of $A[n]$. 
Notice that, since $A[n]$ is finite and stable under the action of $\Aut_0(A)$,
there exists a nontrivial automorphism $\varphi\in \Aut_0(A) $ that acts trivially on $A[n]$.
We claim that the automorphism $\psi$ of $Y \times A$ given by $(y,a) \mapsto (y,\varphi(a))$ 
descends to an automorphism $\overline{\psi}$ of $X$; since $\overline{\psi}$ is not in the 
image of $\iota$, this proves that $\operatorname{Aut}(X)$ is strictly larger than $\iota(A)$.
To see that $\overline{\psi}$ descends to $X$, it suffices to check that for every $t \in A[n]$ we have 
$\psi(t \cdot (y,a)) = t \cdot \psi((y,a))$, that is,
\[
\psi((\iota(t)y,a-t)) = t \cdot (y,\varphi(a)) \Longleftrightarrow (\iota(t) (y),\varphi(a-t)) = (\iota(t) (y),\varphi(a)-t);
\]
this last equality holds since $\varphi$ is a group homomorphism and
$t$ is in $A[n]$, which $\varphi$ fixes pointwise.
\end{proof}

\section{Abelian varieties with finite automorphism group}

We will now prove that any abelian variety such that $\operatorname{Aut}_0(A)$ is finite can be realized as the automorphism group of a smooth projective variety $Y$. We first make some remaks on the structure of abelian varieties with finite automorphism group.

\begin{lemma}\label{lemma:Autfinito}
Let $A$ and $B$ two isogenous abelian varieties then $\Aut_0(A)$ is finite if and only if $\Aut_0(B)$ is finite.
\end{lemma}

\begin{proof}
Since being isogenous is a symmetric relation, it suffices to prove that if $A \to B$ is an isogeny and $\operatorname{Aut}_0(A)$ is infinite, then so is $\operatorname{Aut}_0(B)$. 
Write $B \cong A/H$, where $H$ is a finite subgroup of $A$, and assume that $\Aut_0(A)$ is infinite. 
Notice that every automorphism $\grf$ of $A$ which leaves $H$ stable induces an automorphism $\bar \grf$ of $B$, and that 
$\bar\grf$ is trivial if and only if $\grf$ is trivial. 
Let $n$ be the order of $H$; in particular, we have $H\subset A[n]$. 
Any automorphism $\grf$ of $A$ leaves $A[n]$ stable, so, since $\Aut_0(A)$ is infinite 
and $A[n]$ is finite, the subgroup of automorphisms $\grf$ which fix $A[n]$ pointwise is infinite. 
Every such automorphism leaves $H$ stable, hence it descends to an automorphism of $B$, and since the map $\grf \mapsto \bar{\grf}$ is injective we deduce that $\Aut_0(B)$ is infinite.
\end{proof}

\begin{lemma}\label{lemma:UniquenessSubvarieties}
Let $A$ be an abelian variety such that $\operatorname{Aut}_0(A)$ is finite. Then any two simple abelian subvarieties $A_1, A_2$ of $A$ are isogenous 
if and only if they coincide. Moreover, if $A_1$ is a simple abelian subvariety of $A$, then $\Aut_0(A_1)$ is finite.
\end{lemma}
\begin{proof}
Suppose by contradiction that we can find two distinct but {isogenous} simple abelian subvarieties $A_1, A_2$ of $A$. 
By Poincaré's reducibility theorem, there is an abelian subvariety $C$ of $A$ such that the multiplication map $A_1 \times A_2 \times C \to A$ 
is an isogeny. Let $B$ be an abelian variety such that there exists isogenies $\grf_i:B\lra A_i$ and define the isogeny
\[
\varphi: B^2 \times C \to A \quad\text{ by } \quad \grf(b_1,b_2,c)=\grf_1(b_1)+\grf_2(b_2)+c.
\]
Now notice that $\psi(b_1,b_2,c)=(b_1,b_1+b_2,c)$ defines an automorphism of $B^2\times C$ of infinite order, and by the previous Lemma we
conclude that $\Aut_0(A)$ is also infinite, contradiction.
The proof that for any simple abelian subvariety $A_1$ of $A$ the group $\Aut_0(A_1)$ is finite is completely analogous.
\end{proof}

From now on we fix an abelian variety $A$ with finite automorphism group $\Aut_0(A)$. By the previous Lemma and Poincar\'e reducibility Theorem we know that there 
exist uniquely 
determined simple abelian subvarieties $A_1,\dots,A_h$ of $A$ such that the sum
$$
\grs: A_1\times\cdots \times A_h\lra A  \qquad \grs(a_1,\dots,a_h)=a_1+\dots+a_h 
$$
is an isogeny. We denote by $\Sigma$ the finite kernel of this map and denote by $N$ its order. By Lemma \ref{lemma:UniquenessSubvarieties}, $A_i$ and $A_j$ are not isogenous 
if $i\neq j$ and $\Aut_0(A_i)$ is finite for all $i$. Finally, notice that any abelian variety constructed in this way has finite automorphism group.

\subsection{Construction of the example}
Let $A$ be as above and choose a prime number $p\geq 7$ such that 
\begin{itemize}
 \item[(*)] for $i=1,\dots,h$, for any subgroup $H$ of $A_i$ contained in $A[N]$, and for any nontrivial $\grf\in \Aut_0(A_i/H)$, 
       $p$ is larger than the order of $(A_i/H)^\grf = \{ x \in A_i/H : \grf(x)=x \}$.
\end{itemize}
Notice that if $\grf$ is a nontrivial automorphism of a simple abelian variety then $\grf$ has only finitely many fixed points, so 
a prime number $p$ with this property exists.

Let $S/\mathbb{C}$ be a smooth hypersurface of degree $p$ in $\mathbb{P}^3$ with $\operatorname{Aut}(S) \cong \mathbb{Z}/p\mathbb{Z}$ and such 
that every automorphism of $S$ acts on it without any fixed points; 
an explicit example of such a hypersurface is given in Theorem \ref{thm:AutoS}.
Let $G = \operatorname{Aut}(S) \cong \mathbb{Z}/p\mathbb{Z}$ and set $X := S / G$. 
{We now proceed to describe some basic properties of $X$ (§\ref{sect:PropertiesX}), 
construct a certain smooth projective variety $Y$ of dimension $2+\dim A$ (§\ref{sect:ConstructionY}), 
and prove that $Y$ has automorphism group isomorphic to $A$ (Theorem \ref{thm:AutoY} in §\ref{sect:AutY}).
}

\subsubsection{Properties of $X$}\label{sect:PropertiesX}

\begin{lemma}\label{lemma:XSmoothProjective}
$X$ is a smooth projective variety.
\end{lemma}
\begin{proof}
$X$ is smooth since $G$ acts on $S$ without fixed points, and is projective since any quotient of a projective variety by a finite group of automorphisms is projective (see \cite[Remarque on page 51]{MR0098097}).
\end{proof}
\begin{lemma}\label{lemma:TrivialAutomorphisms}
$X$ does not admit any nontrivial automorphisms.
\end{lemma}
\begin{proof}
Let $\varphi : X \to X$ be an automorphism. Composing with the natural projection $\pi : S \to X$, we obtain a map $\varphi \circ \pi : S \to X$ which, since $S$ is simply connected, lifts to a map $\tilde{\varphi} : S \to S$. Clearly $\tilde{\varphi}$ is algebraic, and it is easily seen to be a covering map, so it is an isomorphism since $S$ is connected and simply connected. It follows that $\tilde{\varphi} : S \to S$ is in $G$, hence (by passing to the quotient) it induces the identity on $X$. Since on the other hand $\tilde{\varphi}$ induces $\varphi$ on $X$, we get $\varphi = \operatorname{id}_X$ as claimed.
\end{proof}

\begin{lemma}\label{lemma:GeneralType}
$X$ has Kodaira dimension $2$.
\end{lemma}
\begin{proof}
Kodaira dimension is invariant under finite étale covers, so $\operatorname{kod}(X)=\operatorname{kod}(S)$. By adjunction, $K_S = O_{\mathbb{P}^3}(p-3-1)|_S$ is ample, so $\operatorname{kod}(S)=\operatorname{dim}(S)=2$.
\end{proof}

\begin{lemma}\label{lemma:MapsToA}
The Albanese variety of $X$ is trivial, therefore there are no non-constant maps from $X$ to any abelian variety.
\end{lemma}
\begin{proof}
Clearly $S$ is the universal cover of $X$, so $\pi_1(X)$ is isomorphic to $\operatorname{Aut}(S \to X) \cong \mathbb{Z}/p\mathbb{Z}$ 
and in particular is finite.
Since the Albanese variety of $X$ is dual to its Picard variety, one has 
$\dim \operatorname{Alb}(X)=\operatorname{dim} \operatorname{H}^1(X,\mathcal{O}_X) =h^{1,0}(X)$; on the other hand, the fact that 
$\pi_1(X)$ is finite implies that $H_1(X,\mathbb{Q})$ is trivial, so $h^{1,0}(X) \leq h^1(X)=\dim H^1(X,\mathbb{C})=0$, hence 
$\operatorname{Alb}(X)$ is trivial as claimed.
\end{proof}

\subsubsection{A nontrivial {$A$-torsor} $Y \to X$}\label{sect:ConstructionY}

\begin{definition}\label{def:Y} Fix an isomorphism $\chi : G \to \mathbb{Z}/p\mathbb{Z}$ and a point $P$ such that 
\begin{itemize}
 \item[(**)] $P$ is a $p$-torsion point of $A$ which is not contained 
in any proper abelian subvariety of $A$. 
\end{itemize}
The abelian subvarieties of $A$ are all of the form $A_{i_1}+\dots+A_{i_k}$, so a point with this property exists.
We let $\mathbb{Z}/p\mathbb{Z}$ act on the group generated by $P$ in the obvious way 
(that is, for $n \in \mathbb{Z}$ the class of $n$ in $\mathbb{Z}/p\mathbb{Z}$ sends $P$ to $nP$).
We set $Y = (S \times A) / G$, where the action of $G$ on the product $S \times A$ is given by
\[
g \cdot (s,a) = (g \cdot s, a+\chi(g)P).
\]
\end{definition}
As in the proof of Lemma \ref{lemma:XSmoothProjective}, it is easy to see that $Y$ is a smooth projective variety; 
moreover, $Y$ has a natural structure of principal space under $A$. Indeed
for each $b \in A$, the translation map
\[
\begin{array}{ccc}
S \times A & \to & S \times A \\
(s,a) & \mapsto & (s,a+b)
\end{array}
\]
commutes with the action of $G$, so it descends to an automorphism of $Y=(S \times A)/G$ that we denote by $y\mapsto b+y$ or by $\tau_b$. 
This defines an action of $A$ on $Y$ which is free and transitive along the fibers of the map $Y\to X$.
Moreover, $Y \to X$ is an $A$-torsor in the analytic (and in fact even étale) topology: indeed, $S$ is an 
étale covering of $X$, and the pullback of $Y$ to $S$ is trivial.

\begin{lemma}\label{lemma:YNonTrivial}
{The map $Y \to X$ does not admit a section (in the analytic topology).}
\end{lemma}
\begin{proof}
Notice that $Y \to X$ {admits a section} if and only if it is trivial as a torsor. 
Indeed if $Y\to X$ has a section $s$ then the map $A\times X\to Y$ given by $(a,x)\mapsto a+s(x)$ is an isomorphism of {torsors}. 
Let $\mathcal{A}$ be the sheaf of holomorphic functions on $X$ with values in $A$; {$A$-torsors on $X$ are classified by $H^1(X,\mathcal{A})$, where the cohomology is taken in the analytic category}. For any fixed $n>0$, consider the exact sequence of sheaves on $X$
\[
0 \to \mathcal{A}[n] \to \mathcal{A} \xrightarrow{[n]} \mathcal{A} \to 0
\]
and take cohomology to obtain the long exact sequence
\[
0 \to H^0(X,\mathcal{A}[n]) \to H^0(X,\mathcal{A}) \xrightarrow{[n]} H^0(X,\mathcal{A}) \to H^1(X,\mathcal{A}[n]) \to H^1(X,\mathcal{A}).
\]
{By Serre's GAGA principle,} all maps from $X$ to $A$ are algebraic, so by Lemma \ref{lemma:MapsToA} we have $H^0(X,\mathcal{A}) = A$, and $H^0(X,\mathcal{A}) \xrightarrow{[n]} H^0(X,\mathcal{A})$ is just $A \xrightarrow{[n]} A$, which is surjective. It follows in particular that the natural arrow 
\begin{equation}\label{eq:injectivityGeneral}
H^1(X,\mathcal{A}[n]) \to H^1(X,\mathcal{A})
\end{equation}
is injective. 
Consider $Z := (S \times \langle P \rangle) / G \hookrightarrow Y$, where $\langle P \rangle$ denotes the order $p$ subgroup of $A(\mathbb{C})$ generated by $P$.
By the injectivity of \eqref{eq:injectivityGeneral} (with $n=p$), proving that $Z$ is a nontrivial covering space of $X$ suffices to show that $Y \to X$ is a nontrivial {torsor}. But this is 
clear, because the natural map $S \to S \times A \to (S \times \langle P \rangle) / G$ is injective and surjective, hence (since $S$ is compact) a homeomorphism. It follows 
that $Z \cong S$ is a nontrivial cover of $X$ as desired.
\end{proof}

\subsection{Determination of $\operatorname{Aut}(Y)$}\label{sect:AutY}
In this section we show:

\begin{theorem}\label{thm:AutoY}
The automorphism group of $Y$ is isomorphic to $A$.
\end{theorem}

\subsubsection{Preliminaries on simple abelian varieties}
We shall need the following basic fact about simple abelian varieties.

\begin{lemma}\label{lemma:DeterminantOnCohomology}
Let $T$ be a projective complex torus. Let $A$ be the abelian variety obtained from $T$ by fixing an arbitrary origin; notice that $T$ is naturally a torsor under $A$. Finally let $\gra$ be an automorphism of $T$ (as a projective variety) and assume that $A$ is simple. Then:
\begin{enumerate}
   \item if $\alpha$ is translation by a point of $A$, then the determinant of $(1-\alpha)_* : H_1(T,\mathbb{Q}) \to H_1(T,\mathbb{Q})$ is $0$;
   \item if $\alpha$ is \textit{not} translation by a point of $A$, then $\gra$ has at least one fixed point and the determinant of 
         $(1-\alpha)_* : H_1(T,\mathbb{Q}) \to H_1(T,\mathbb{Q})$ is the number of fixed points of $\alpha$.
\end{enumerate}
\end{lemma}
\begin{proof}
The statement of (1) is obvious, because translations induce the identity on $H_1(T,\mathbb{Q})$.
Assume now that $\gra$ is not a translation and identify $T$ with $A$ by choosing a point $t_0 \in T$ as the origin. We prove first that $\gra$ has at least one fixed point. 
Letting $a=\gra(t_0)-t_0$ we have $\gra(t)=\grf(t)+a$, where $\grf\in \Aut_0(A)$ is different from the identity. Let $\psi=\grf-\operatorname{id}_A :A\lra A$; it is an endomorphism of $A$, and since $A$ is simple and $\grf$ is nontrivial the image of $\psi$
is $A$ itself. One checks that $b \in T$ is a fixed point of $\gra$ if and only if $\psi(b)=-a$. As $\psi$ is surjective, such $b$ exist, and there are only finitely many of them because the set $\{ b : \psi(b)=-a\}$ is naturally a torsor under the finite group $\ker \psi$.
We can then choose the origin $t_0$ to be a fixed point of $\gra$, in which case $\gra$ belongs to $\operatorname{Aut}_0(A)$ and we have $\psi(t)=\gra(t)-t$, 
so that $A^\grf$ is equal to the kernel of $\psi$ and its order is the degree of $\psi$. The lemma follows from the fact that
for a complex torus $H_n(\psi,\mathbb Q)=\det(\psi_*:H_1(T,\mathbb Q)\lra H_1(T,\mathbb Q))$.
\end{proof}

\subsubsection{Preliminaries on surfaces of Kodaira dimension 2}
We shall need the following consequence of \cite{MR2393263}.

\begin{lemma}\label{prop:ImagesOfAbelianVarieties}
Let $S$ be a surface of Kodaira dimension 2 and $A$ be an abelian variety. 
The image of any morphism $f : A \to S$ is either a point or a (possibly singular) irreducible curve of geometric genus at most one.
\end{lemma}
\begin{proof}
Suppose by contradiction that $f$ is surjective. Then by \cite[Theorem 1.1]{MR2393263} the surface $S$ admits a finite étale cover which is a product of projective spaces and an abelian variety: since Kodaira dimension is invariant under finite étale covers, this contradicts the fact that the Kodaira dimension of $S$ is 2, because any such product has non-positive Kodaira dimension. So the image of $f$ can only be a single point or a curve, which is then automatically irreducible since $A$ is. Suppose that the image of $f$, call it $Z$, is a curve, and let $\tilde{Z}$ be its normalization. By the universal property of normalization, $f$ 
induces a (dominant, hence surjective) map from $A$ to $\tilde{Z}$; applying \cite[Theorem 1.1]{MR2393263} again we obtain that $\tilde{Z}$ is covered by either $\mathbb{P}^1$ or an elliptic curve, which proves the statement about the genus.
\end{proof}

\subsubsection{Proof of Theorem \ref{thm:AutoY}}
We already noticed that $A$ injects into $\operatorname{Aut}(Y)$.
For the other inclusion let $\varphi$ be an automorphism of $Y$. We prove first that $\grf$ preserve the fibers of the map $\pi:Y\lra X$. 
For each $x \in X$, let $Y_x$ be the fiber of $\pi$ over $x$ and let
\[
\varphi_x : Y_x \hookrightarrow Y \xrightarrow{\varphi} Y \xrightarrow{\pi} X.
\]
Suppose that for general $x$ the image of $Y_x$ is not reduced to a single point: then Lemma \ref{prop:ImagesOfAbelianVarieties} implies 
that generically the image of $\varphi_x$ is a (possibly singular) curve of genus at most 1. By \cite[Proposition VII.2.1]{MR2030225},
a surface of Kodaira dimension 2 admits no algebraic system (of positive dimension) of effective divisors whose general member is a 
(possibly singular) rational or elliptic curve. By Lemma \ref{lemma:GeneralType} we know that $X$ is a surface of Kodaira dimension 2, so it follows that $\varphi_x$ is 
constant for all $x \in X$. In particular, 
\[
Y \xrightarrow{\varphi} Y \to A\backslash Y = X
\]
descends to a map $\varphi_X : X \to X$, which is easily seen to be biregular (its inverse being $(\varphi^{-1})_X$), and hence an automorphism. 
It follows from Lemma \ref{lemma:TrivialAutomorphisms} that $\varphi_X$ is the identity, which implies that the equality $\varphi (Y_x)=Y_x$ 
holds for all $x \in X$. Thus we see that for every $x \in X$ the automorphism $\varphi$ of $Y$ induces an automorphism $\varphi |_{Y_x}$ of $Y_x$. 

\medskip

Thus, locally in the analytic topology, the automomorphism $\varphi$ can be described as follows. 
For each $x\in X$ we can choose an open connected neighborhood $U \subset X$ of $x$ such that 
$V=\pi^{-1}(U)$ can be identified with $U\times A$ {(as an $A$-torsor)} and $\grf(u,a)=(u,\phi(u,a))$. 
Let $r:U\lra A$ be defined by $r(u)=\phi(u,0)-0$. Then 
$a\mapsto \phi(u,a)-r(u)$ is an automorphism of $A$ as an algebraic group, 
and since $\Aut_0(A)$ is finite it must be equal to an automorphism $\phi$ independent of $u$.
Hence {$\varphi(u,a)=(u,\phi(a)+r(u))$} and $\psi=\phi-id_A$ is an endomorphism of $A$ as an algebraic group. 
{Furthermore, since any two identifications of a fiber of $Y \to X$ with the trivial $A$-torsor differ only by a translation, 
we see that the endomorphism $\psi$ thus obtained is independent of our choice of $U$ and of the local trivialization $\pi^{-1}(U) \cong U \times A$.
}

\medskip

We now prove the theorem by induction on $h$, the number of simple factors of $A$. 
Assume first that $h=1$, so that $A$ is simple. For $x\in X$ we define 
$$n(x) = \det \left( (1-\varphi|_{Y_x})_* \bigm\vert H_1(Y_x,\mathbb{Q}) \right);$$
it is a continuous function on $X$. 
Since $X$ is connected and $\mathbb{Z}$ is discrete, it follows that $n(x)$ is actually constant: 
let $n$ be the common value of the various $n(x)$. We show that $n=0$.
Suppose by contradiction that $n > 0$. Let $\tilde X=Y^\grf$ and let $\tilde \pi$ be the restriction of $\pi$ to $\tilde X$. 
We prove that $\tilde \pi$ is an $n$-to-1 covering of $X$. The fact that it is $n$-to-1 follows from Lemma \ref{lemma:DeterminantOnCohomology}.
The claim that it is a covering can be checked locally using the analytic topology: using the local description above we obtain 
$V^\grf=\{(u,a): \psi(a)=r(u)\}$, which is a covering of $U$.

If at least one of the connected components of $\tilde{X}$ is the trivial cover of $X$, then this gives a section of the projection map $\pi : Y \to X$, 
contradicting Lemma \ref{lemma:YNonTrivial}. 
Otherwise, take a connected subcover of $\tilde{X}$: this is a connected $m$-to-1 cover of $X$ for some $m \leq n$ which is smaller than $p$ by our assumption 
(*) on $p$. This contradicts the fact that $\#\pi_1(X)=p$. 

It follows that $n(x)=n=0$ for all $x$, hence by Lemma \ref{lemma:DeterminantOnCohomology} $\varphi|_{Y_x}$ is translation by a point $a(x) \in A$ 
(recall that $Y_x$ is naturally a torsor under $A$, so it makes sense to identify translations of $Y_x$ with elements of $A$). Now $x \mapsto a(x)$ 
gives a map $X \to A$, which is necessarily constant by Lemma \ref{lemma:MapsToA}, hence $\varphi$ is globally a translation by a point of $A$. 

\medskip

We now prove the inductive step. Let $h>1$. 
Since $\grf$ preserves the fibers $Y_x$, composing with a translation by an element of $A$ we can assume that there exists 
$y_0 \in Y$ such that $\grf(y_0)=y_0$. We want to prove that in this case $\grf$ is the identity.
Let $\pi : A \to A':=A/A_1$ be the natural projection and set $A_i' := \pi(A_i)$ for $i=2,\ldots,h$. We let $P'=\pi(P)$ and write $\tilde{\pi} : A_1 \times \cdots \times A_h \to A_2' \times \cdots \times A_h'$ for the homomorphism
\[
\tilde{\pi}(a_1,\ldots,a_h) = (\pi(a_2), \ldots, \pi(a_h));
\]
finally, we set $\Sigma' := \tilde{\pi}(\Sigma)$. One then checks that the sum $\sigma' : A_2' \times \cdots \times A_h' \to A'$ is an isogeny with kernel $\Sigma'$.

Let $K = \ker(\Sigma \to \Sigma')$. For every $i=2,\ldots,h$, the intersection $A_1 \cap A_i$ embeds naturally into $K$,
so $N' \cdot \#(A_1 \cap A_i) \bigm\vert N' \cdot \#K = N$. It follows that every quotient of $A_i'=A_i/(A_1 \cap A_i)$ by a subgroup of $A_i'[N']$ is a quotient of $A_i$ by a subgroup of $A_i[N]$, so the analogue of condition (*) is satisfied by $A'$ and the prime $p$. It is immediate to check that (**) also holds for $A'$, $p$, and the point $P'$.
In particular, by induction, the automorphism group of $Y'=S\times ^G A'$ is equal to $A'$.

The projection map $S\times A\lra S\times A'$ is $G$-equivariant, so it induces a map $q:Y\lra Y'$ which we prove to be a 
categorical quotient by the action of $A_1$. Indeed let $f:Y\lra Z$ be a $A_1$-invariant map. 
It induces a $G\times A_1$-invariant map $f_1:S\times A\lra Z$ 
and therefore a map $f_2:S\times (A_1\times\dots\times A_h) \lra Z$ which is invariant by the action of {both $A_1$ and $\Sigma$ on the second factor.
Since the quotient of $A_1 \times \cdots \times A_h$ by the subgroup generated by $A_1$ and $\Sigma$ is $A'$,} the map $f_2$ induces a 
regular map {$g_2:S\times A'\lra Z$ such that $f_2=g_2\circ(id_S\times \pi')$, where $\pi' := \pi \circ \sigma$ is the 
natural map $A_1 \times \cdots \times A_h \to A'$. Since furthermore $f_2$} is $G$-invariant, {$g_2$} is also $G$-invariant, 
hence it induces a map $g:Y'\lra Z$ such that $f=g\circ q$. Moreover, as $q$ is surjective, the map $g$ is unique.

We can now prove that $\grf$ is the identity. For $a\in A$ denote by $\tau_a$ the translation by $a$ in $Y$. 
Notice that for each $a$ and for each $x\in X$ there exists $\phi_x \in \Aut_0(A)$ such that
$$ \grf\circ \tau_a \circ \grf^{-1}=\tau_{\phi_x(a)}:Y_x\lra Y_x. $$
In particular, if $a\in A_1$, then $\phi_x(a)\in A_1$. 
Being $Y'$ a categorical quotient of $Y$ by the action of $A_1$, we have that 
$\grf$ induces a map $\grf':Y'\lra Y'$, which is an automorphism since $(\grf^{-1})'$ is its inverse. Moreover, the image of $y_0$ in $Y'$ is 
fixed by $\grf'$, so $\grf'$ is equal to the identity. 

Hence $\grf(y)-y \in A_1$ for all $y\in Y$. Arguing in the same way, but using $A_2$ instead of $A_1$, we obtain $\grf(y)-y \in A_2$ for all $y$.
So $\grf(y)-y\in A_1\cap A_2$ for all $y\in Y$, and since $A_1\cap A_2$ is finite and $\grf(y_0)=y_0$ we obtain $\grf(y)=y$ for all $y$. \qed

\subsection{A hypersurface in $\mathbb{P}^3$ with automorphism group $\mathbb{Z}/p\mathbb{Z}$}\label{sect:Hypersurface}In this section we explicitly construct, for every prime $p \geq 7$, an algebraic surface in $\mathbb{P}^3$ of degree $p$ whose automorphism group is cyclic of order $p$:
\begin{theorem}\label{thm:AutoS}
Let $p \geq 7$ be a prime number, and for $\lambda \in \mathbb{C}$ let $S_\lambda$ be the algebraic surface over $\mathbb{C}$ given by the zero locus in $\mathbb{P}^3$ of the homogeneous polynomial
\[
f_\lambda(x_1,x_2,x_3,x_4) := x_1^p + x_2^p + x_3^p + x_4^p + \lambda ( x_1^2 x_2^{p-4} x_3^2 + x_1^4 x_2^{p-6} x_4^2 ).
\]
The surface $S_\lambda$ is smooth for all but finitely many $\lambda \in \mathbb{C}$; if $\lambda \neq 0$, the automorphism group of $S_\lambda$ is cyclic of order $p$, generated by $[x_1 : x_2 : x_3 : x_4] \mapsto [x_1 : \zeta_p x_2 : \zeta_p^2 x_3 : \zeta_p^3 x_4]$, where $\zeta_p$ is a primitive $p$-th root of unity.
Moreover, each nontrivial element of $\operatorname{Aut}(S_\lambda)$ acts on $S_\lambda$ without any fixed points.
\end{theorem}

We start by noticing that for $\lambda=0$ the surface $S_0$ is smooth. Since being smooth is a Zariski-open condition in the defining polynomial, this shows that $S_\lambda$ is smooth away from a proper Zariski-closed subset of $\mathbb{C}$, that is, $S_\lambda$ is smooth for all but finitely many values of $\lambda$. From now on fix a nonzero value of $\lambda$ such that $S_\lambda$ is smooth, and to simplify the notation write $S$ for $S_\lambda$ and $f(x_1,x_2,x_3,x_4)$ for $f_\lambda(x_1,x_2,x_3,x_4)$.

By \cite[Theorem 2]{MR0168559} we know that all the automorphisms of $S$ are induced by 
(linear) automorphisms of $\mathbb{P}^3$, so we only need to consider these.
Let $L : \mathbb{P}^3 \to \mathbb{P}^3$ be a linear transformation that satisfies 
$L(S)=S$. We identify $L$ to the class $[M] \in \operatorname{PGL}_4(\mathbb{C})$ of a matrix $M=(M_{ij}) \in \operatorname{GL}_4(\mathbb{C})$. Furthermore, we let $e_1,\ldots,e_4$ be the canonical basis of $\mathbb{C}^4$ and denote by $\langle e_i \rangle$ the 1-dimensional $\mathbb{C}$-vector subspace of $\mathbb{C}^4$ generated by $e_i$.
We shall show Theorem \ref{thm:AutoS} in three steps: first we shall prove that $M$ either fixes or permutes the lines generated by $e_3$ and $e_4$; then we shall show that the same statement holds for the lines generated by $e_1$ and $e_2$; finally, we shall deduce from this that $M$ needs to be a diagonal matrix, at which point a direct computation concludes the proof. {This approach is inspired by \cite{MR2129679}.}

\subsubsection{Step 1: $M$ permutes $\langle e_3 \rangle$ and $\langle e_4 \rangle $}
The condition that $L(S)=S$ translates into the polynomial equality
\begin{equation}\label{eq:Automorphism}
f \circ M(x_1,\ldots,x_4) = \alpha f(x_1,\ldots,x_4)
\end{equation}
for some $\alpha \in \mathbb{C}^\times$.
Applying $\frac{\partial^2}{\partial x_i \partial x_j}$ to the two members of this equation and setting 
\[
H_{ij}(x_1,\ldots,x_4) := \frac{\partial^2f}{\partial x_i \partial x_j}(x_1,\ldots,x_4)
\]
we find
%
%
%
%
\[
\sum_k \sum_m  M_{kj} M_{mi} H_{mk} (M(x_1,\ldots,x_4)) = \alpha H_{ij}(x_1,\ldots,x_4).
\]
Let $u, v$ be two vectors in $\mathbb{C}^4$. Multiplying the previous identity by $u_iv_j$ and summing over $i$ and $j$ we get
\begin{equation}\label{eq:SecondDerivatives}
\sum_{k,m} (Mv)_k (Mu)_{m} H_{mk}( M(x_1,\ldots,x_4)) = \alpha\sum_{i,j} H_{ij}(x_1,\ldots,x_4) u_i v_j.
\end{equation}
We now define a bilinear pairing
\[
\begin{array}{cccc}
\langle \cdot, \cdot \rangle : & \mathbb{C}^4 \times \mathbb{C}^4 & \to & \mathbb{C}[x_1,\ldots,x_4] \\
& (u,v) & \mapsto & \sum_{i,j} H_{i,j}(x_1,\ldots,x_4) u_i v_j,
\end{array}
\]
so that Equation \ref{eq:SecondDerivatives} reads
\[
\langle Mu, Mv \rangle \left(  M(x_1,\ldots,x_4) \right) = \alpha \langle u,v \rangle.
\]
In particular, since $M$ is invertible we obtain:
\begin{proposition}\label{prop:Orthogonality}
Let $u, v$ be vectors in $\mathbb{C}^4$. The equalities $\langle u,v \rangle=0$ and $\langle Mu, Mv \rangle=0$ are equivalent.
\end{proposition}

\begin{lemma}\label{lemma:Orthogonal}
Let $a, b \in \mathbb{C}^4$ be two nonzero vectors such that $\langle a,b \rangle=0$. Then there exist $\lambda, \mu \in \mathbb{C}^\times$ such that either $a=\lambda e_3, b=\mu e_4$, or $a = \lambda e_4, b = \mu e_3$ hold.
\end{lemma}

\begin{proof}
Write $a=(a_1,a_2,a_3,a_4)$ and $b=(b_1,b_2,b_3,b_4)$. 
By direct inspection, one checks that, for $i=1,2,3,4$, the only second derivative of $f$ involving the monomial $x_i^{p-2}$ is $H_{ii}$. This immediately implies that $a_ib_i=0$ for $i=1,\ldots,4$, and by symmetry we can assume $a_1=0$.
The coefficients of the monomials $x_1 x_2^{p-5} x_3^2$, $x_1 x_2^{p-4} x_3$ and $x_1^3 x_2^{p-6} x_4$ in $\langle a,b \rangle$ are given by
$2 \lambda (p-4) (a_2 b_1 + a_1 b_2), 4 \lambda (a_3 b_1 + a_1 b_3)$ and $8 \lambda (a_4 b_1 + a_1 b_4)$ respectively, so under our assumptions $\langle a,b \rangle=0$, $\lambda \neq 0$ and $a_1=0$ we obtain $b_1a_2=b_1a_3=b_1a_4=0$. If we had $b_1 \neq 0$, this would imply $a=(0,0,0,0)$, contradicting our assumptions, so we must have $b_1=0$ as well. The situation is now again symmetric in $a,b$, so we might assume $a_2=0$. Arguing as before (but looking at the monomials $x_1^2 x_2^{p-5} x_3$ and $x_1^4 x_2^{p-7} x_4$) one finds $a_3b_2=a_4b_2=0$, so that $b_2=0$ as well. The conclusion now follows easily from the equalities $a_3b_3=a_4b_4=0$.
\end{proof}


\begin{corollary}\label{cor:Me1e2}
One of the following holds:
\begin{itemize}
\item $M \langle e_3 \rangle = \langle e_3 \rangle$ and $M \langle e_4 \rangle =\langle e_4 \rangle$;
\item $M \langle e_3 \rangle = \langle e_4 \rangle$ and $M \langle e_4 \rangle =\langle e_3 \rangle$.
\end{itemize}
\end{corollary}
\begin{proof}
Apply Proposition \ref{prop:Orthogonality} to $u=e_3$ and $v=e_4$: since $\langle e_3,  e_4\rangle=H_{34}=0$ we obtain
$
\langle M e_3, M e_4 \rangle = 0
$. The claim then follows from the previous lemma.
\end{proof}

\subsubsection{Step 2: $M$ permutes $\langle e_1 \rangle$ and $\langle e_2 \rangle$}

Arguing as in the previous section, it is easily seen that if we let $A : (\mathbb{C}^4)^p \to \mathbb{C}$ denote the multilinear form
\[
A: (u_1,\ldots,u_p) \mapsto \sum_{i_1,\ldots,i_p} \frac{\partial^p f}{\partial x_{i_p} \cdots \partial x_{i_p}} 
(u_1)_{i_1} \cdots (u_p)_{i_p},
\]
where $(u_i)_j$ is the $j$-th coordinate of $u_i$,
we have $A(Mu_1,\ldots,Mu_p) = \beta A(u_1,\ldots,u_p)$ for some $\beta \in \mathbb{C}^\times$; 
notice that here we do not need to compose with $M$ on the left 
hand side, because $p$-th derivatives of $f$ are just scalars. Suppose that $M \langle e_3 \rangle = \langle e_3 \rangle$ and $M \langle e_4 \rangle= \langle e_4 \rangle$; the case $M \langle e_3 \rangle = \langle e_4 \rangle$ and $M \langle e_4 \rangle = \langle e_3 \rangle$ is completely analogous. Rescaling $M$ if necessary (which we can do, since we are only interested in its projective class) we can assume $Me_3=e_3$.
Choosing $u_1=\cdots=u_{p-1}=e_3$ and $u_p=e_1$ we have
\[
\beta A(e_3,\ldots,e_3,e_1) = \beta \frac{\partial^p f}{\partial x_3^{p-1} \partial x_1}=0,
\]
from which we deduce
\[
0 = A(Me_3,\ldots,Me_3,Me_1) = A(e_3,\ldots,e_3,Me_1)= \sum_{i_p} \frac{\partial^p f}{\partial x_3^{p-1} \partial x_{i_p}} (Me_1)_{i_p};
\]
since the only nonvanishing partial derivative of the form $\frac{\partial^p f}{\partial x_3^{p-1} \partial x_{i_p}}$ is $\frac{\partial^p f}{\partial x_3^p}$, this implies $M_{31}=0$. Similary, the choice $(e_4,\ldots,e_4,e_1)$ shows $M_{41}=0$, while the choices $(e_3,\ldots,e_3,e_2)$ and $(e_4,\ldots,e_4,e_2)$ give $M_{32}=M_{42}=0$.
It follows that $M$ sends the 2-plane $\{x_3=x_4=0\}$ to itself; in particular, $M$ induces an automorphism of the finite set of points in $\mathbb{P}^3$ defined by the equations
\[
f(x_1,x_2,x_3,x_4)=0, \quad x_3=x_4=0 \quad \Longleftrightarrow \quad x_3=x_4=0, \quad x_1^p+x_2^p=0.
\]
From this it is immediate to deduce: 
\begin{corollary}\label{cor:Me3e4}
One of the following holds:
\begin{itemize}
\item $M \langle e_1\rangle = \langle e_1 \rangle$ and $M \langle e_2 \rangle = \langle e_2 \rangle$;
\item $M \langle e_1 \rangle = \langle e_2 \rangle$ and $M \langle e_2 \rangle = \langle e_1 \rangle$.
\end{itemize}
\end{corollary}

\subsubsection{Step 3: determination of $\operatorname{Aut}(S)$}
Corollaries \ref{cor:Me1e2} and \ref{cor:Me3e4}
tell us that $M$ either fixes or permutes the lines $\langle e_1 \rangle, \langle e_2 \rangle$, and that the same holds for the lines $\langle e_3 \rangle, \langle e_4 \rangle$.
One checks easily that if $M$ exchanges $\langle e_1 \rangle$ with $ \langle e_2 \rangle$, and/or it exchanges $\langle e_3 \rangle$ with $\langle e_4 \rangle$, then $f \circ M$ is not a scalar multiple of $f$, so that $M$ needs to be a diagonal matrix. Normalize $M$ so that $M_{11}=1$ and write $M=\operatorname{diag}(1,\mu_2,\mu_3,\mu_4)$: replacing in Equation \eqref{eq:Automorphism} and comparing the coefficients of $x_1^p$ on the two sides we find $\alpha=1$. Comparing the coefficients of $x_i^p$ for $i=2,3,4$ we then obtain $\mu_i^p=1$ for $i=2,3,4$, so that $\mu_2,\mu_3,\mu_4$ are $p$-th roots of unity. It is now immediate to check that the only automorphisms of $S$ are represented by the powers of the (order $p$) matrix
$
\begin{pmatrix}
1 \\ & \zeta_p \\ & & \zeta_p^2 \\ & & & \zeta_p^3
\end{pmatrix},
$
where $\zeta_p$ is a primitive $p$-th root of unity. The fixed points (in $\mathbb{P}^3$) for the action of this matrix (or any of its powers, with the exception of the identity) are $[1:0:0:0], [0:1:0:0], [0:0:1:0], [0:0:0:1]$, none of which lies on the hypersurface $f(x_1,x_2,x_3,x_4)=0$.
This concludes the proof of Theorem \ref{thm:AutoS}.

\bibliographystyle{alpha}
\bibliography{Biblio}
\end{document}